\documentclass[10pt]{article}
\usepackage[utf8]{inputenc}
\usepackage{amsthm,geometry,amsmath,amsfonts,amssymb,graphicx,mdwlist,subfigure,dsfont,bbold}
\theoremstyle{plain}
\newtheorem{theorem}{Theorem}[section]
\newtheorem{corollary}[theorem]{Corollary}
\newtheorem{lemma}[theorem]{Lemma}
\newtheorem{proposition}[theorem]{Proposition}

\theoremstyle{remark}
\newtheorem{remark}{Remark}

\theoremstyle{definition}
\newtheorem{definition}{Definition}

\def\Im{\text{Im}\,}
\def\Re{\text{Re}\,}
\def\uno{\mathbb 1}
\def\0{\mathit 0}

\def\CC{\mathbb C}
\def\RR{\mathbb R}
\def\MM{\mathbb M}
\def\multi#1{{\scriptsize \begin{array}{c} #1 \end{array}}}
\def\t{\text{\itshape{\textsf{T}}}}
\def\({\left(}	\def\){\right)}
\def\K{{\mathcal K}} \def\A{{\mathcal A}}
\def\sqmatrix#1#2{\left[\begin{array}{#1} #2 \end{array}\right]}
\def\me{\mathcal }
\def\ts#1{\textstyle{#1}}

\begin{document}

	\author{Francesco Tudisco \\ \footnotesize{Department of Mathematics and Computer Science,}\\ \footnotesize{Saarland University, Saarbr\"ucken, Germany.}}
 \date{ }

	\title{A note on certain ergodicity coefficients}

\maketitle
	\abstract{We investigate two ergodicity coefficients $\phi_{\|\, \|}$ and $\tau_{n-1}$,   originally introduced to bound the subdominant eigenvalues of nonnegative matrices. 
	The former has been generalized to complex matrices in recent years and several properties for such generalized version have been shown so far. 
	We provide a further result concerning the limit of its powers. Then we propose a generalization of the second coefficient $\tau_{n-1}$ and we show that, under mild conditions, it can be used to recast the eigenvector problem $Ax=x$ as a particular M-matrix linear system, whose coefficient matrix can be defined in terms of the entries of $A$. Such property turns out to generalize the two known equivalent formulations of the Pagerank centrality of a graph. \\
	
	
	}
%
%

\section{Introduction}
In this note we  discuss measures of
ergodicity from the point of view of linear algebra and in the context of finite Markov chains. Loosely speaking a Markov chain involves a sequence of stochastic matrices, and the chain is ergodic if the product of such matrices converges. 
Coefficients of ergodicity were introduced to estimate how fast, if at all, these
products converge. In the simplest case all factors in the product are identical to the same stochastic
matrix $A$, and, in this case, ergodicity coefficients have been used essentially to bound the subdominant eigenvalues of $A$. Indeed, when $A$ is stochastic, the powers $A^k$ converge to a stochastic matrix all of whose rows are equal, with a convergence rate depending on the magnitude of the subdominant eigenvalue $|\lambda_2|$, whose estimation is therefore extremely interesting.

According to  \cite[\S 1]{seneta-explicit}, ergodicity coefficients can be subdivided into two classes: The first one consists  of coefficients defined in terms of a vector norm $\|\, \|$, maximized over a suitable subspace; the second one is the class of coefficients  defined in terms of a deflation of the original matrix. Consider a nonnegative irreducible matrix $A$,  let $x> \0$ be its dominant eigenvector $Ax = \rho(A)x$ and let $\lambda_2$ denote an eigenvalue of $A$, different from $\rho(A)$, having largest absolute value.  
The following scalar functions
\begin{equation}\label{int1}
\mu_{\| \, \|}(x,A) = \max_{\multi{\|y\|\leq 1 \\ y^\t x = 0 \\ y \in \CC^n}}\|y^\t A\|, \qquad   \tau_{\| \, \|}(x,A)= \max_{\multi{\|y\|\leq 1 \\ y^\t x = 0 \\ y \in \RR^n }}\|y^\t A\|\, ,
\end{equation}
are an example of ergodicity coefficients belonging to the first class. To our knowledge $\mu_{\|\, \|}$ and $\tau_{\|\, \|}$ were introduced  
in \cite{seneta} and \cite{rothblum1}, respectively. A number of relevant properties are investigated in those papers, in particular it is shown that 
\begin{gather}
	|\lambda_2|\leq \mu_{\|\, \|}(x,A)\, , \qquad |\lambda_2|\leq \tau_{\|\, \|}(x,A)\, ,\label{int2} \\
	|\lambda_2|=\lim_{k \to \infty}\mu_{\|\, \|}(x,A^k)^{1/k}\, , \qquad |\lambda_2|=\lim_{k \to \infty}\tau_{\|\, \|}(x,A^k)^{1/k}\, .\label{int3}
\end{gather}
 The coefficients in \eqref{int1} have been studied and extended by various authors afterwards (see e.g.\ \cite{Hartfiel1999,rothblum3,ergodicity-survey,Rhodius1997}). In particular Ipsen and Selee consider in \cite{ergodicity-survey}  a generalized version of the coefficients in \eqref{int1}, defined for any complex matrix $A$ by\footnote{That coefficient is actually denoted in \cite{ergodicity-survey} using the same symbol $\tau_{\| \, \|}$ used in the first pioneering works. We use, though, a different notation for the sake of clearness.} 
	$$\phi_{\| \, \|}(W, A) = \max_{\multi{x \in \CC^n\\ \|x\|\leq 1 \\ W^*x =\0}} \|x^* A\|\,\,  ,$$
being $W$ a given (possibly rectangular) matrix. Several properties of $\phi_{\| \, \|}$ are shown in \cite{ergodicity-survey}, in particular it is proved that $\phi_{\| \, \|}$ provides an upper bound for the absolute value of a generic eigenvalue of $A$ \cite[Thm.\ 7.4]{ergodicity-survey}, extending and improving the results in \eqref{int2}. 
In Section \ref{sec:complex-tau}  we give a short overview of the results proposed in \cite{ergodicity-survey}, then we show in Theorem \ref{gen2}  that the modulus of a generic eigenvalue of a complex matrix $A$ can be represented as a limit involving $\phi_{\|\, \|}$, obtaining that also the property \eqref{int3} carries over to  $\phi_{\|\, \|}$. We conclude the section discussing the relation among $\mu_{\|\, \|}$, $\tau_{\|\, \|}$ and $\phi_{\|\, \|}$,   observing in particular that  properties \eqref{int2} and \eqref{int3} can be easily recovered as a consequence of the proposed results.

In Section \ref{sec:linear-system} we analyze another coefficient of ergodicity which was defined in \cite{rothblum2} for nonnegative matrices $A$ having a positive dominant left eigenvector $y$, and there used to bound the second largest eigenvalue of $A$. This coefficient is an example of ergodicity coefficient belonging to the second class, as indeed it can be related to a deflation of the considered matrix $A$ (c.f. \cite{rothblum2}). Its original definition is
\begin{equation*}
\tau(A,y)=\rho(A) - \sum_{i=1}^n \(\min_{j=1,\dots,n}y_j^{-1}a_{ij}\)y_i\, .
\end{equation*}
 We obtain an alternative formula for $\tau(A,y)$, which suggests a possible generalization that may be useful  to estimate the other eigenvalues of $A$. Moreover we show in Theorem \ref{lin-sis} that when $A$ is column stochastic, $\tau(A,y)$ can be used to prove that the dominant right eigenvector of $A$ (i.e.\ $x$ such that $Ax=x$) can be computed by solving a suitable M-matrix linear system of equations,  whose coefficient matrix is explicitly defined in terms of  the entries of $A$. When applied to the Google matrix $G$ associated with the web graph, Theorem \ref{lin-sis} implies the well known result that the Pagerank centrality problem $p^\t = p^\t G$ is equivalent to a linear system problem.
\subsubsection*{Notation}\label{notations}
We briefly fix here our notation. Given two integers $n$ and $m$, the symbol $\MM_{n,m}(\mathbb F)$  denotes the set of  $n\times m$ matrices over the field $\mathbb F$. When $n=m$ we shorten the notation by writing $\MM_{n}(\mathbb F)$ in place of $\MM_{n,n}(\mathbb{F})$. For our purposes $\mathbb  F$ will be either $\CC$ or $\RR$. For $A \in \MM_n(\mathbb F)$ let $\sigma(A)$ be the set of its eigenvalues and $\rho(A) = \max_{\lambda \in \sigma(A)}|\lambda|$ be its spectral radius. 
The symbol $\uno$ denotes the vector of all ones, $\0$ denote the vector of all zeros and $O$ the zero matrix. \\
A square matrix $A$  is called 
\begin{itemize}
	\item \textit{nonnegative}, if its entries are nonnegative numbers, in symbols $A\geq O$ (analogous notation is used for vectors).
	\item \textit{positive}, if its entries are positive numbers, in symbols $A>O$ (analogous notation is used for vectors).
	\item \textit{reducible}, if there exists a permutation $P$ such that $PAP^{\t}$ is block triangular
	\item \textit{irreducible}, if it is not reducible.
\end{itemize}
\vspace{-8pt}
We shall make freely use of the Perron-Frobenius theory for nonnegative matrices. 

\section{Ergodicity coefficients for complex matrices}\label{sec:complex-tau}
Given any $n\times n$ complex matrix $A$, let $\lambda_1, \lambda_2,\dots, \lambda_s$ be its distinct eigenvalues,  ordered so that 
$$|\lambda_1|\geq |\lambda_2|\geq \cdots \geq |\lambda_s|\, .$$
\begin{definition}
	Given $A \in \MM_n(\CC)$, a vector norm  $\| \, \|$, and any $W \in \MM_{n,t}(\CC)$,  we define
	$$\phi_{\| \, \|}(W, A) = \max_{\multi{x \in \CC^n\\ \|x\|\leq 1 \\ W^*x =\0}} \|x^* A\|\, .$$
\end{definition}

In \cite[\S 7]{ergodicity-survey} Ipsen and Selee give a generalization of some known results \cite{rothblum1,seneta}, concerning the second largest eigenvalue of a nonnegative matrix, essentially extending to complex matrices, and thus to $\phi_{\|\, \|}$, several properties  of the coefficient of ergodicity $\tau_{\| \, \|}$ 
firstly proposed
by Seneta in \cite{seneta} and then deeply investigated by Rothblum and Tan in \cite{rothblum1}. In particular Theorem 3.3 in the latter paper shows  that the modulus of the second largest eigenvalue of a nonnegative irreducible matrix can be expressed in terms of a limit involving the powers of $\tau_{\| \, \|}$. In this section we summarize the properties of $\phi_{\|\, \|}$ which have been proved in \cite[\S 7]{ergodicity-survey}, by collecting them into the following Theorem \ref{teo:ipsen-seele}, and then we complete such theorem by proving for complex matrices the limit property shown in \cite[Thm.\ 3.3]{rothblum1} (see also \eqref{int3}) for real matrices .

We recall that $\phi_{\| \, \|}$ is a continuous scalar function and is probably the most general form of an ergodicity coefficient defined by a vector norm. It is often denoted by $\tau_{\|\, \|}$ (\cite[\S 7]{ergodicity-survey}, \cite[\S 7]{rothblum1} or \cite[\S 3]{rothblum3} for instance) but here we prefer to use a different symbol for the sake of clearness, whereas we use the symbol $\tau_{\| \, \|}$ to denote the ``original'' coefficient introduced in \cite[p.\ 59]{rothblum1}.

\index{$\phi_{\| \, \|}$}
\begin{theorem}\label{teo:ipsen-seele} 
The coefficient $\phi_{\|\, \|}$ is bounded and both well-conditioned and weakly submultiplicative in the second argument. Precisely, let $A,A_1, A_2 \in \MM_n(\CC)$ and let $W \in \MM_{n,t}(\CC)$, then
\begin{enumerate}
\item $0 \leq \phi_{\| \, \|}(W, A) \leq \|A^*\|$.
\item $|\phi_{\| \, \|}(W,A_1)-\phi_{\| \, \|}(W,A_2)| \leq \phi_{\| \, \|}(W,A_1-A_2)$.
\item $\phi_{\| \, \|}(W,A_1A_2)\leq \|A_2^*\|\phi_{\| \, \|}(W,A_1)$.
    \item For any $X \in \MM_{t,n}(\CC)$ it holds that $\phi_{\| \, \|}(W,A)\leq \|A^*-X^*W^*\|$.
\item If $W$ spans the Jordan space associated to the first $\lambda_1, \dots, \lambda_t$ distinct eigenvalues of $A$, then
    $$|\lambda_{t+1}|\leq \phi_{\| \, \|}(W,A)\, .$$
\end{enumerate}
\end{theorem}
As the above theorem is a collection of results taken from \cite{ergodicity-survey}, its proof is omitted. Actually  points \emph{1, 2, 3, 4} are not difficult to prove, whereas \emph 5 is proved in a slightly more general and selfcontained  form in the forthcoming Lemma \ref{gen1}.

In what follows $\K$ is any closed bounded subset of $\CC^n$ with nonempty interior and containing properly the origin (i.e.\ it does not lie on the boundary of $\K$). Also,  $\| \, \|$ is any norm on $\CC^n$. We shall need a couple of important lemmas.


\begin{lemma}\label{moz}
	Let $A \in \MM_n(\CC)$,  $\K\subseteq \CC^n$ be any compact set that contains properly the origin, and $W_t$ the matrix whose columns span the Jordan subspace associated with the eigenvalues $\lambda_1, \dots, \lambda_t$ of $A$.
	Then
	$$|\lambda_{t+1}| \geq \limsup_{k\to \infty}\Bigl(\max_{\multi{x\in \K \\ W_t^* x =\0}}\|x^* A^k\|\Bigr)^{\frac 1 k}\, .$$
\end{lemma}
\begin{proof}
	 Let $J_t$ and $J_t'$ be  the direct sums of the elementary Jordan blocks corresponding to $\lambda_i$, respectively for $1\leq i \leq t$ and $t+1\leq i \leq s$, being $s$ the number of distinct eigenvalues of $A$. Then  the   Jordan
	 canonical form of $A$  can be assumed to be as follows
	 $$A = X J X^{-1}=\sqmatrix{cc}{W_{t} & W_t'}\sqmatrix{cc}{J_t & \\
	 	& J_t'}\sqmatrix{cc}{W_t & W_t'}^{-1}\, .$$
	Set $\Omega = \K \cap \ker W_t^*$. Now, if $x \in \Omega$ then
	$$\textstyle{x^* A^k = x^* X J^k X^{-1} =\sqmatrix{c|c}{\0 & \, x^*W_t' }J^k X^{-1}}\, ,$$
	since by definition $x^* W_t=\0$. It is clear therefore that replacing the Jordan block $(J_t)^k$
	inside $J^k$ with the zero matrix, does not affect the product $x^* A^k$.  Now let $\mu > |\lambda_{t+1}|$ and define the functions
	$$f_k:\CC^n \to \RR_+, \quad f_k(x)= \left\| \frac{x^* A^k}{\mu^k}\right\|\, .$$
	Note that $f_k(x)$ converges pointwise to zero for any $x \in \Omega$. Also note that
	$x_0^{(k)} = \arg \(\sup_{x \in \Omega}f_k(x)\right)$ belongs to  $\Omega$,
	therefore $f_k(x)$ converges uniformly as well. As a consequence $(f_k(x))_k$ is bounded
	in $\Omega$, i.e.\ there exists $M=M_\mu>0$ such that
	$$\forall \mu >|\lambda_{t+1}|, \qquad \sup_{x \in \Omega}\left\| \frac{x^* A^k}{\mu^k}\right\| = \sup_{\Omega}f_k\leq M\, .$$
	Now, taking   the $\limsup$ over $k$, we observe that,  for any $\mu>|\lambda_{t+1}|$, it holds
	\begin{align*}
		\limsup_{k \to \infty} \(\max_{x\in \Omega}\|x^* A^k\|\)^{1/k}&= \limsup_{k \to \infty} \(\max_{x\in \Omega}\left\|\frac{x^* A^k}{\mu^k}\right\|\mu^k\)^{1/k}\\
		&\leq \limsup_{k\to \infty} (M\mu^k)^{1/k}=\mu\, .
	\end{align*}
	The thesis thus follows by continuity allowing $\mu$ to decrease to $|\lambda_{t+1}|$.
\end{proof}

\begin{lemma}\label{gen1} Let $A\in \MM_n(\CC)$, let $\K$ be any compact set in $\CC^n$ containing the origin, and let $W_t$ be the matrix whose columns span the Jordan subspace associated with the  eigenvalues $\lambda_1, \dots, \lambda_t$ of $A$.  Given any $p$ columns $w_1, \dots, w_p$ of $W_t$, $p\geq 1$, set
	$W_{t, p} = \sqmatrix{ccc}{w_1 & \cdots & w_p}$. Then
	$$\(\min_{y \in \partial \K}\|y\|\)\, \cdot \,|\lambda_{t+1}|^k  \leq \max_{\multi{x \in
		\K \\ W_{t,p}^*x=\0} }\|x^* A^k\| \, , \qquad\quad  \forall \, k\geq 1$$
	being $\partial \K$  the boundary of $\K$. 
\end{lemma}
\begin{proof}
  It is easy to see that the left (right) eigenvectors of an eigenvalue $\lambda$ are orthogonal to the right (left) eigenvectors of any other eigenvalue $\mu$, such that $\mu\neq \lambda$. Therefore the set $\partial \K \cap \ker W_{t,p}^*$  contains a left eigenvector $u$ of $A$ relative to $\lambda_{t+1}$, that is there exists $u \in \partial \K \cap \ker W_{t,p}^*$ such that $u^* A = \lambda_{t+1} u^*$.  
  
   The proof is now straightforward since
	$$|\lambda_{t+1}|^k \, \min_{y\in \partial \K}\|y\| = \min_{y \in \partial \K}\|\lambda_{t+1}^k\,  y\|\leq\|\lambda_{t+1}^k \, u\|$$
	and of course $u^* A^k = \lambda_{t+1}^k u^*$.
\end{proof}

We are ready to state the limit property for $\phi_{\|\,\|}$ that we announced.
%
\begin{theorem}\label{gen2}
	Let $A \in \MM_n(\CC)$ and let $W_t$ be the matrix whose columns span the Jordan space  relative to the eigenvalues $\lambda_1, \dots, \lambda_t$ of $A$. Then
	$$|\lambda_{t+1}| = \lim_{k \to \infty}\phi_{\| \, \|} (W_t,A^k)^{1/k}\, .$$
\end{theorem}

\begin{proof}
	The inequality $|\lambda_{t+1}|\leq \liminf_k\phi_{\| \, \|}
	(W_t,A^k)^{1/k}$ follows by Lemma \ref{gen1}  when $\K$ is the unit ball $\K = \{z \in \CC^n : \|z\|\leq 1\}$,  $p$ is the  number of columns in $W_t$, and then taking the infimum over $k$.   Analogously, by Lemma \ref{moz}, the same choice for $\K$ gives $|\lambda_{t+1}|\geq \limsup_k \phi_{\| \, \|}
	(W_t,A^k)^{1/k}$. Combining the two bounds we conclude the proof.
\end{proof}

\subsubsection*{On the relation among  $\mu_{\|\, \|}$, $\tau_{\|\, \|}$ and $\phi_{\|\, \|}$}
 Given an irreducible nonnegative matrix $A$, Rothblum, Seneta and Tan   considered in \cite{rothblum1,seneta} the  following scalar functions
\begin{equation*}
\mu_{\| \, \|}(x,A) = \max_{\multi{\|y\|\leq 1 \\ y^\t x = 0 \\ y \in \CC^n}}\|y^\t A\|, \qquad   \tau_{\| \, \|}(x,A)= \max_{\multi{\|y\|\leq 1 \\ y^\t x = 0 \\ y \in \RR^n }}\|y^\t A\|\, ,
\end{equation*}
where $x$ is a real positive right eigenvector of $A$ corresponding to the eigenvalue $\rho(A)$, i.e.\ $Ax = \rho(A)x$, $x> \0$.  The function $\phi_{\|\,\|}$ is a generalized version of both $\mu_{\|\,\|}$ and $\tau_{\|\,\|}$. In order to support this claim we devote the remaining part of this section to show that several properties proved for $\mu_{\|\,\|}$ and $\tau_{\|\, \|}$ in \cite{rothblum1,seneta} can be derived from Lemma \ref{gen1} and Theorem \ref{gen2}. Although some of the considerations that follow are quite easy, to our knowledge none of them have been stated in explicit form before.

In \cite[pp.\ 579-580]{seneta} Seneta observed that $\mu_{\| \,\|}(x,A)$ is a bound for the second largest modulus of the eigenvalues of $A$. One can easily derive this fact by noting that $\phi_{\| \, \|}(W_{1},A)\leq \mu_{\| \, \|}(x,A)$, for any nonnegative matrix $A$, where $W_1$ spans the Jordan space associated with $\rho(A)$. The equality holds when $A$ is irreducible, as in that case the Perron-Frobenius theorem implies that $W_1=x$.

More precisely, given any $A \geq O$, Lemma \ref{gen1} and Theorem \ref{gen2} imply the following inequalities
\begin{align}
|\lambda_2|&\leq \phi_{\| \, \|}(W_{1},A^k)^{1/k}\leq \mu_{\| \|}( x,A^k)^{1/k}, \quad \forall k \geq 1\, \, ,\label{7}\\
|\lambda_2|&=\lim_{k\to \infty}\phi_{\| \, \|}(W_{1}, A^k)^{1/k}\leq \lim_{k\to \infty}\mu_{\| \, \|}(x,A^k)^{1/k}\, ,\label{2}
\end{align}
being $\lambda_1, \lambda_2, \lambda_3, \dots$ the distinct eigenvalues of $A$. Again, equality holds in \eqref{2} when $A$ is irreducible.

Rothblum and Tan \cite[Thms.\ 3.1, 3.3, 7.1]{rothblum1} have shown that  the same  properties \eqref{7} and \eqref{2} hold for $\tau_{\|\,\|}$. Inspiring our argument to the approach they used, here we  draw their result from Theorem \ref{gen2} and the following further Lemma \ref{gen1.1}.

Given any vector $x \in \CC^n$, decompose it as the sum $x = \Re x + \mathbf i\,  \Im x$,  where $\Re x $ and $\Im x$ are the unique real vectors defined in the obvious way by the real and imaginary parts of the entries of $x$. Consider the following norm over $\CC^n$
$$\nu(x)= \sup_{\alpha \in \RR}\|\Re x \cos \alpha + \Im x \sin \alpha\|\, ,$$
Note that $\nu$ is a norm if and only if $\| \, \|$ is a norm.
\begin{lemma}\label{gen1.1}
	Let $A\in \MM_n(\RR)$ and let $W\in \MM_{n,p}(\RR)$. Then 
	 $$\max_{\multi{x \in \CC^n \\ \nu(x) \leq 1\\ W^\t x =\0}}\nu(x^* A)=\max_{\multi{x \in \RR^n \\ \|x\| \leq 1\\ W^\t x =\0} }\|x^\t A\|\, .$$
\end{lemma}

\begin{proof}
	Since $W$ is real, if  $W^\t x = \0$ then both  $\Re x$ and $\Im x$ belong to $\ker W^\t$ and $(\Re x \cos \alpha + \Im x \sin \alpha) \in \ker W^\t$ as well. As a consequence, for any $x \in \ker W^\t$ such that $\nu(x) \leq 1$, we have
	$$ \nu(x^* A)=\sup_{\alpha \in \RR}{\|(\Re x \cos \alpha + \Im x \sin \alpha)^\t A\|}\leq \max_{\multi{z \in \RR^n \\ \|z\| \leq 1\\ W^\t z =\0}}\|z^\t A\|\, ,$$
	thus by taking the maximum over $x \in \CC^n$ such that $ W^\t x = \0$ and $\nu(x) \leq 1$ we get 
	$$\max_{\multi{x \in \CC^n \\ \nu(x) \leq 1\\ W^\t x =\0}}\nu(x^* A)\leq \max_{\multi{x \in \RR^n \\ \|x\| \leq 1\\ W^\t x =\0} }\|x^\t A\|\, .$$
	As the reverse inequality is straightforward, the proof is complete.
\end{proof}
It is now not difficult to obtain the analogous of \eqref{7} and \eqref{2} for $\tau_{\|\, \|}$. Assume that $A$ is real and that the matrix $W_t$, spanning the Jordan space relative to the first $t$ distinct eigenvalues of $A$ has $p$ real columns. By using Lemma \ref{gen1}, with $\K=\{x \in \CC^n : \nu(x)\leq 1\}$,  and Lemma \ref{gen1.1}, we get 
	\begin{equation}\label{88}
|\lambda_{t+1}|\leq \phi_{\nu}(W_{t,p},A)= \max_{\multi{x \in \RR^n \\ \|x\| \leq 1\\ W^\t_{t,p} x =\0} }\|x^\t A\|\, ,
	\end{equation}
	where $W_{t,p}$ is the matrix made by the $p$ real columns of $W_t$ and, as before, $\lambda_1, \lambda_2, \lambda_3, \dots$ are the distinct eigenvalues of $A$. 
	In particular, if $A \geq O$ and $x$ is a real nonnegative right eigenvector associated to $\rho(A)$, then \eqref{88} and Theorem \ref{gen2} imply
	\begin{align}
	|\lambda_2|&\leq \phi_{\nu}(W_{1},A^k)^{1/k}\leq \tau_{\| \|}( x,A^k)^{1/k}, \quad \forall k \geq 1\, \, ,\notag\\
	|\lambda_2|&=\lim_{k\to \infty}\phi_{\nu}(W_{1}, A^k)^{1/k}\leq \lim_{k\to \infty}\tau_{\| \, \|}(x,A^k)^{1/k}\, ,\label{22}
	\end{align}	
	where $W_1$ spans the Jordan space of $A$ relative to the eigenvalue $\rho(A)$. As for $\mu_{\|\,\|}$, if $A\geq O$ is irreducible, then $W_1 = x$ and the equality holds in \eqref{22}. 
\section{Nonnegative matrices with a positive dominant eigenvector}\label{sec:linear-system}
In this section we slightly change the notation: given a $n\times n$ matrix $A$, we let $\lambda_i$, $1\leq i \leq n$ denote its eigenvalues, this time counting multiplicities. As before we assume them ordered  according to decreasing magnitude
$$|\lambda_1|\geq |\lambda_2|\geq \cdots \geq |\lambda_n|\, .$$

Consider now a $n\times n$ nonnegative matrix\footnote{Recall that  $\rho(A)=\lambda_1\geq |\lambda_2|\geq \dots\geq |\lambda_n|$ for any $n\times n$ nonnegative matrix  $A$.} $A\neq O$ and let $y\geq \0$ be a left dominant eigenvector, i.e.\ $y^\t A = \rho(A)y^\t$.  When $y>\0$ we may consider another functional $\tau(A,y)$  given by 
\begin{equation}\label{t}
\tau(A,y)=\rho(A) - \sum_{i=1}^n \(\min_{j=1,\dots,n}y_j^{-1}a_{ij}\)y_i\, .
\end{equation}
Haviv, Ritov and Rothblum introduced that functional in \cite{rothblum2} and showed therein that $\tau(A,y)$  has properties somehow analogous to those of $\tau_{\|\,\|}$. In fact (see \cite[Thm.\ 1]{rothblum2})
\begin{equation}\label{s}
|\lambda_2|\leq \tau(A,y) \quad \text{and} \quad |\lambda_2|=\lim_{k\to \infty}\tau(A^k,y)^{1/k}\, .
\end{equation}

Now let $\me S^+$ be the set of nonnegative column stochastic  matrices, that is $S \in \me S^+$ if and only if $S\geq O$ and $\uno^\t S=\uno^\t$, and let $\A^+$ be the set of nonnegative and nonnull matrices having a positive dominant left eigenvector. Note that any matrix $A \in \A^+$ is such that $\rho(A)>0$.  It is known that any $A \in \A^+$ is similar, via a positive definite diagonal similarity, to a  nonnegative matrix whose columns sum up to $\rho(A)$ (see for instance\ \cite{stochastic-similar-doubly,tudisco-power-nonnegative}). Thus
$$\A^+=\{\mu D\me S^+ D^{-1} :  \mu>0, \, \, D \text{ diagonal,}\, \, d_{ii}>0 \, \, \forall i \}\, .$$
As a consequence the spectrum of any  $A \in \A^+$ is the spectrum of an element of $\me S^+$ scaled by a positive factor. It follows that, if $y>\0$ is a left dominant eigenvector of $A \in \A^+$, and $S\in \me S^+$ is such that $DSD^{-1}=\rho(A)^{-1}A$ (note that such an $S$ may be not uniquely defined since  the geometric multiplicity of $\rho(A)$ can be larger than $1$  in general), then
$$\tau(A,y)=\rho(A)\tau(DSD^{-1},y)=\rho(A)\tau_{n-1}(S)\, ,$$
where $\tau_{n-1}(S)=\tau(S,\uno)=1-\sum_i \min_j s_{ij}$.

So we reduce our attention to $\me S^+$ and to \index{$\tau_{n-1}$} $\tau_{n-1}:\me S^+ \to \RR$ rather than on \eqref{t}. Such $\tau_{n-1}$ is sometimes denoted by $\beta$ (see for instance  \cite[p.\ 157]{ergodicity-survey} or \cite[p.\ 137]{seneta-markov-chains}) and is one of the  earlier \textit{proper} ergodicity coefficients. A proper ergodicity coefficient was originally defined as a continuous scalar function $f$ from $\me S^+$ to the real interval $[0,1]$ and such that $f(M)=0$ if and only if $M$ is a rank-one matrix in $\me S^+$ \cite[p.\ 509]{seneta-historical}. For the sake of completeness, and in order to justify our choice of the symbol $\tau_{n-1}$, we state the following proposition, which nonetheless will be  used to prove Theorem~\ref{lin-sis}.
\begin{proposition}\label{tau-obs}Let $A$ be a matrix of $\me S^+$. Then
	\begin{itemize}
		\item  $\tau_{n-1}(A) = \max_{j=1,\dots,n}\sum_{i\in V}a_{ij}-\sum_{i\in V}\min_{k=1,\dots,n}a_{ik}$, for any subset $V\subset \{1,\dots,n\}$ such that $|V|=n-1$.
		\item $\tau_{n-1}:\me S^+ \to [0,1]$. In particular $\tau_{n-1}(A)=0$ if and only if $A=x \uno^\t$, for some vector  $x\geq \0$, $x^\t \uno =1$.  Whereas $\tau_{n-1}(A)=1$ if and only if each row of $A$ has at least one zero entry (or, in other words, $\tau_{n-1}(A)<1$ if and only if $\max_i \min_j a_{ij}>0$).
	\end{itemize}
\end{proposition}

\begin{proof}
	Let $A\in \me S^+$ and let $V=\{1,\dots,n\}\setminus \{h\}$, then $\max_j \sum_{i \in V}a_{ij}= 1-\min_j a_{hj}$. This proves the first statement. A direct inspection easily shows $0\leq \tau_{n-1}(A)\leq 1$ and $\tau_{n-1}(A)=1 \iff \max_i \min_j a_{ij}=0$. Finally, if $\tau_{n-1}(A)=0$ then we see from its definition that $1 = \sum_i a_{ij}=\sum_i \min_h a_{ih}$ for all $j=1,\dots,n$. This implies $a_{ij} = \min_{h}a_{ih}$ $\forall j$, therefore the rows of $A$ must be entrywise constant, i.e.\ $A=x \uno^\t$, for some $x\geq \0$, $x^\t \uno =1$. The reverse implication is straightforward.
\end{proof}

Therefore we can rewrite $\tau_{n-1}(A)$ as the maximum of a certain function of the entries of $A$ over the subsets $V\subset \{1,\dots,n\}$ of cardinality $n-1$. This fact immediately suggests a generalization to $\tau_m$, given by
\begin{equation}\label{taum}
\tau_{m}(A)=\max_{\multi{V\subset\{1,\dots,n\}\\ |V|=m}}\max_{j \in \{1,\dots,n\} }\sum_{i\in V}\Bigl(a_{ij}-\min_{k \in \{1,\dots, n\}} a_{ik}\Bigr)\, ,
\end{equation}
where $m=1,\dots, n-1$. Observe that
\begin{proposition}\label{tau-pro}For any $A \in \me S^+$, if $k\leq m$ then $\tau_k(A)\leq \tau_m(A)$.
\end{proposition}

\begin{proof}
	Let $V_k\subset \{1,\dots,n\}$, $|V_k|=k$, be the subset realizing $\tau_k(A)$. For any $V_m \supseteq V_k$ we have
	\begin{align*}
	\tau_k(A)&= \ts{\max_j} \sum_{i\in V_k}(a_{ij}-\ts{\min_h} a_{ih})\\
	&\leq  \ts{\max_j} \left\{\sum_{i \in V_k}(a_{ij}-\ts{\min_h} a_{ih})+\sum_{i \in V_m\setminus V_k}(a_{ij}-\ts{\min_h} a_{ih})\right\}\\
	&= \ts{\max_j }\sum_{i \in V_m}(a_{ij}-\ts{\min_h} a_{ih})\\
	&\leq \tau_m(A)
	\end{align*}
	since $a_{ij}\geq \min_h a_{ih}$ for any $i$ and $j$.
\end{proof}

Let $1 = \lambda_1 \geq |\lambda_2|\geq \cdots \geq |\lambda_n|$ be the eigenvalues of $A \in \me S^+$. 
Let us point out  that  $\tau_1$, as $\tau_{n-1}$, is a proper coefficient of ergodicity. Indeed it is not difficult to  observe that  $\tau_1(A)=0$ if and only if $A$ has all constant rows, that is $\tau_{1}(A)=0$ if and only if $A$ is a rank one matrix $A=x\uno^\t$ where $x\geq \0$ and $x^\t \uno =1$. We obtain in turn that $\tau_1(A)=0$ implies $\lambda_n=0$, therefore the inequality $|\lambda_n|\leq \tau_1(A)$ holds when either $A$ is singular (as in that case $|\lambda_n|=0$)  or when $\tau_1(A)$ is zero. We have moreover
\begin{proposition}\label{prosym}
	If $A \in \me S^+$ is symmetric, then $|\lambda_n|\leq \tau_1(A)$.
\end{proposition}
\begin{proof}
	As $A$ is real symmetric, the quantity 
	$\lambda = \min_{x \neq 0}\frac{x^\t A x}{x^\t x} $ belongs to $\sigma(A)$. Let $1\leq h,k\leq n$ be distinct indices such that $a_{hh}\geq a_{kk}\geq 0$, and let $e_i$ denote the $i$-th canonical vector. Then
	\begin{align*}
		|(e_h-e_k)^\t A (e_h-e_k)|&=|a_{hh}+a_{kk}-2\, a_{hk}|\leq 2\,  |a_{hh}-a_{hk}|\\
		&\leq 2 \, (\ts{\max_{j}}a_{hj}-\ts{\min_{r}}a_{hr})\\
		&\leq 2\,\ts{\max_{i}}(\ts{\max_{j}}a_{ij}-\ts{\min_{r}}a_{ir})=2\, \tau_1(A)\, .
	\end{align*}
	Therefore, as $(e_h-e_k)^\t (e_h-e_k)=2$ for any $h\neq k$, we have 
	\begin{equation*}
	|\lambda_n|\leq |\lambda|=\left|\min_{x \neq 0}\frac{x^\t A x}{x^\t x} \right|\leq \left|\frac{(e_h-e_k)^\t A (e_h-e_k)}{2} \right|\leq \tau_1(A)\, .\qedhere
	\end{equation*} 
\end{proof}

To our opinion it is reasonable to guess at this stage that the inequality 
\begin{equation}\label{eq:guess}
	|\lambda_{k}|\leq \tau_{n-k+1}(A)
\end{equation} holds for several values of $k\in \{2,\dots, n\}$. There are indeed a number of evidences in support of such a guess: Proposition \ref{tau-pro} shows that $\tau_{n-k+1}(A)$, as $|\lambda_k|$,  is monotonically non-increasing with respect $k$, moreover the suggested inequality \eqref{eq:guess} holds for  $k=2$ and any $A \in \me S^+$ (see \eqref{s} and the discussion after Lemma \ref{alpha}). Finally, Proposition \ref{prosym} and the considerations immediately before show that,   when $A$ is either singular or symmetric, or $\tau_1(A)=0$, then \eqref{eq:guess} holds also for $k=n$. This leaves us with the nontrivial open question to determine those matrices $A\in \me S^+$ and those indices $3 \leq k \leq n-1$ for which \eqref{eq:guess} is satisfied. 
We firmly think that further investigations on this direction would be of significant interest. 

\begin{remark}
Given $m \in \{1, \dots, n-1\}$, consider the following further coefficient 
\begin{equation*}
\tilde{\tau}_{m}(A)=\min_{\multi{V\subset\{1,\dots,n\}\\ |V|=m}}\max_{j \in \{1,\dots,n\} }\sum_{i\in V}\Bigl(a_{ij}-\min_{k \in \{1,\dots, n\}} a_{ik}\Bigr), 
\end{equation*}
obtained by replacing the maximum over $V$ with the minimum over $V$  in the definition \eqref{taum} of $\tau_m$. 
It is straightforward to observe that $\tilde \tau_{m}(A)\leq \tau_m(A)$, for all $m=1, \dots, n-1$, moreover $\tilde{\tau}_{n-1}(A)=\tau_{n-1}(A)\geq |\lambda_2|$, due to Proposition \ref{tau-obs}, and $\tilde \tau_1(A)=0$ implies $|\lambda_n|=0$, as any column stochastic matrix with a constant row must be singular. Therefore one could in principle use $\tilde{\tau}_m$ to generalize the ergodicity coefficient $\tau_{n-1}$, rather than $\tau_m$, as we propose. Note furthermore that $\tilde{\tau}_m$, as $\tau_m$, is monotonically non-decreasing with respect to the index $m$, that is the statement of Proposition \ref{tau-pro} holds unchanged if $\tau_m$ is replaced by $\tilde{\tau}_m$ (the proof can be easily obtained by an argument analogous to the one used in Proposition \ref{tau-pro}, but starting with the subset $V_m$ realizing $\tau_m(A)$ and then considering any $V_k\subseteq V_m$). 
However, the magnitude of the eigenvalues $\lambda_3,\lambda_4,\dots$ may happen to be larger than the numbers $\tilde{\tau}_{n-2}(A), \tilde{\tau}_{n-3}(A), \dots$, respectively. The following $3\times 3$ matrix provides an example: 
$$A = \sqmatrix{ccc}{
	0 & 0.29 & 0.55\\
	0.63 & 0.4 & 0.12\\
	0.37 & 0.31 & 0.33} \quad \qquad  \begin{tabular}{ccc}
$|\lambda_2| = 0.3341$ & $\tau_2(A)=0.57$ & $\tilde{\tau_2}(A)=0.57$ \\
$|\lambda_3| = 0.0641$ & $\tau_2(A)=0.55$ & $\tilde{\tau_1}(A)=0.06$\\
\end{tabular}$$
Note instead that for such $3\times 3$ matrix the inequality \eqref{eq:guess} holds for both $k=2$ and $k=3$. We want to point out that we tested \eqref{eq:guess} on many random  matrices, generated according to several different distributions (e.g. uniform, normal, preferential attachment, small world), but  we was not able to find any matrix $A\in \me S^+$ and any integer $2\leq k \leq n$ such that \eqref{eq:guess} does not hold.
%
\end{remark}


\subsection{The computation of the dominant eigevector}
Given $A \in \me S^+$ let  $x$ be a right nonnegative dominant eigenvector of $A$, such that $\uno^\t x = 1$. In this subsection we show that, under mild conditions and using $\tau_{n-1}$, $x$ can be characterized as the solution of a linear system of the form
$$Mx=y\, ,$$
where $M$ is an M-matrix, and the entries of both $M$ and $y$ are explicitly defined in terms of the elements of $A$.

\begin{lemma} \label{alpha}
	Let $A \in \me S^+$ and $\alpha \in \RR$. Given $h\in \{1,\dots,n\}$ consider the vector $y=y(\alpha)$ such that $y_i=\min_j a_{ij}$, $i\neq h$, $y_h = 1-(\sum_{i\neq h}y_i) -\alpha$, and note that $\uno^\t y = 1-\alpha$. Then
	\begin{enumerate}
		\item The columns of $\frac 1 \alpha (A-y  \uno^\t)$ sum up to one, for any $\alpha\neq 0$,
		\item $C=A-y \uno^\t$ is nonnegative, for any $\alpha \geq \tau_{n-1}(A)$.
	\end{enumerate}
\end{lemma}
\begin{proof}
	(1) By definition we have $\frac 1 \alpha (A-y \uno^\t)^\t \uno =\frac 1 \alpha (\uno-(1-\alpha)\uno)=\uno$. (2) Simply observe that if $i \neq h$, then  $c_{ij}=a_{ij}-\min_{k}a_{ik}\geq 0$ and
	$$c_{hj}=a_{hj}-y_h = 1-\sum_{i\neq h}a_{ij}-(1-\alpha -\sum_{i\neq h}y_i)\geq \alpha-\tau_{n-1}(A)\geq 0$$
	for $j=1,\dots,n$.
\end{proof}

Observe that due to the previous lemma, for any nonnegative column stochastic $A$ and any $\alpha \geq \tau_{n-1}(A)$, we have the decomposition $A = \alpha B +y \uno^\t$, where $B$ is still nonnegative column stochastic. As a direct consequence of this fact we obtain an alternative proof of the well known upper bound in \eqref{s}, namely:  $|\lambda|\leq \tau_{n-1}(A)$, for any $\lambda \in \sigma(A)$ such that $\lambda \neq 1$. In fact, given $\uno$, let $\{\uno,y_2,\dots,y_n\}$ be a set of linearly independent vectors and let $Y$ be the nonsingular matrix whose rows are such vectors. Then
$$YAY^{-1}=\alpha YBY^{-1}+Yy \uno^\t Y^{-1}=\alpha \sqmatrix{cc}{1 & \0^\t \\ w & Q}+ \sqmatrix{cc}{1-\alpha & \0^\t \\ z & O}\, ,$$
which implies $\sigma(A)= \{1\}\cup \sigma(\alpha Q)$. Note that for all $\lambda \in \sigma(\alpha Q)$ we have $|\lambda|\leq \alpha$, since the eigenvalues of $Q$ are eigenvalues of $B$. The claimed bound now follows by taking the infimum over $\alpha$.  

Recall that any irreducible nonnegative column stochastic matrix $A$ with $\sigma(A)\setminus \{1\}\subseteq \{z \in \CC : |z|<1\}$ is such that $A^k>O$ for some $k\geq 1$. Then observe that,  combining this observation with  Proposition \ref{tau-obs}, one gets the following interesting consequence, whose simple proof is omitted for brevity
\begin{corollary}
	Let $A$ be a nonnegative irreducible column stochastic  matrix such that $\textstyle{\max_i \min_j a_{ij}>0}$. Then there exists $k\geq 1$ such that $A^k >O$.
\end{corollary}
The main theorem of this section now follows
\begin{theorem}\label{lin-sis} 
	Given  $A \in \me S^+$ let $y$ be the vector $y_i=\min_j a_{ij}$ and $B\in \me S^+$ the matrix $\tau_{n-1}(A)B =A-y \uno^\t$.  If $y$ is not the zero vector and $x$ is a right nonnegative dominant eigenvector of $A$ s.t.\ $\uno^\t x=1$, then $I-\tau_{n-1}(A)B$ is nonsingular and
$$x = (I-\tau_{n-1}(A)B)^{-1}y\, \, .$$
Viceversa if $I-\tau_{n-1}(A)B$ is invertible then the vector $y_i = \min_j a_{ij}$ is nonzero and $x = (I-\tau_{n-1}(A)B)^{-1}y$ is a right nonnegative dominant eigenvector of $A$ such that $\uno^\t x =1$.
\end{theorem}
\begin{proof}
	If $y \neq \0$ then $\max_i \min_j a_{ij}>0$ implying that $\tau_{n-1}(A)<1$ and, thus, $I-\tau_{n-1}(A)B$ is invertible. Moreover, by computing the product $(I-\tau_{n-1}(A)B)x$, and using the equalities $Ax=x$ and $\uno^\t x =1$, we get $(I-\tau_{n-1}(A)B) x=y$. Viceversa, if  $I-\tau_{n-1}(A)B$ is invertible, then $\tau_{n-1}(A)\neq 1$ (i.e.\ $\tau_{n-1}(A)<1$), thus by Proposition \ref{tau-obs} $\max_i\min_ja_{ij}>0$, that is $y$ is nonzero and   $x=\sum_{k\geq 0}\tau_{n-1}(A)^kB^k y \geq \0$. Moreover,  since $\uno^\t y = 1-\tau_{n-1}(A)$,
	$$\uno^\t x = \uno^\t \sum_{k\geq 0}\tau_{n-1}(A)^k B^ky=\sum_{k\geq 0}\left\{\tau_{n-1}(A)^k-\tau_{n-1}(A)^{k+1}\right\}=1\, .$$
	Now expand the equation  $(I-\tau_{n-1}(A)B)x=y$ to observe that $x$ is a nonnegative fixed point of $A$. 
\end{proof}	

Theorem \ref{lin-sis} is interesting from a computational point of view. Indeed it shows that the eigenvector problem $Ax=x$ can be interchanged with the linear system problem $(I-\tau_{n-1}(A)B)x=y$, when dealing with nonnegative  stochastic matrices $A$. In other words, even though the problem of computing the eigenvector $Ax=x$ is in general substantially different with respect the problem of solving a linear system, Theorem \ref{lin-sis} shows that when $\tau_{n-1}(A)\neq 1$, or equivalently $\max_i \min_j a_{ij}>0$, or equivalently $y$ is not the zero vector, then actually we can compute the eigenvector $x$ as the solution of the system of linear equations  $(I-\tau_{n-1}(A)B)x=y$, allowing the use of possibly any linear system solver to approximate $x$. 

Theorem \ref{lin-sis} has  possibly several applications. To our opinion a  relevant example is the Google's Pagerank index problem where  the particular structure of the model allows to recast the stationary distribution problem in terms of a linear system (\cite{pagerank-ipsen-selee,google} e.g.). 
The Google engine web matrix (or Pagerank transition matrix), here denoted by $G=(g_{ij})$, is a convex combination of a row stochastic matrix $T$ (the transition matrix of the graph) and a rank-one row stochastic matrix: $G=c T  +(1-c)\uno v^\t$, where $v$ is a positive  vector whose elements sum up to one, and $0< c < 1$. Due to the very high dimension of $G$, various algorithms essentially based on the power method have been proposed to compute the stationary distribution $p$ such that  
\begin{equation}\label{eq:prk}
 p^\t = p^\t G\, ,
\end{equation}
as for instance in \cite{BZ07,BZS05,IK06,extrapolation}. 
As the asymptotic convergence of the power method 
depends on  the magnitude of the subdominant eigenvalue of $G$
, ergodicity coefficients are strongly related with such approach. Several authors have investigated this relation in details \cite{secondoautoval,IK06,seneta-markov,jordan_google}. On the other hand, the original formula by S.\ Brin and L.\ Page \cite{BP98} defines the Pagerank vector $p$ as the solution of a M-matrix linear system of the type
\begin{equation}\label{eq:pagarank-system}
\gamma(I-c T)^\t  p = v, \quad \gamma \in \RR\, ,
\end{equation}
usually referred to as the Pagerank system. We can  recover that linear system by means of Theorem \ref{lin-sis}. Assume for the sake of simplicity that each column of $T$ has at least one zero entry (i.e.\ that there is no node in the graph pointed by all nodes). Since $\max_i \min_j (G^\t)_{ij}=(1-c)\max_i v_i >0$ then 
$$p = (I- \tau_{n-1}(G^\t)B)^{-1}y\, ,$$
where $\tau_{n-1}(G^\t)=1-\sum_i \min_j (G^\t)_{ij}=c$, $y = (1-c)v $ and $B = \frac 1 c (G^\t-y\uno^\t)=T^\t$. This shows, indeed, that $p$ is both the solution of the eigenvector problem \eqref{eq:prk} and of the  linear system \eqref{eq:pagarank-system}, with $\gamma = \frac{1}{1-c}$.

Starting from the  linear system  formulation of the Pagerank problem \eqref{eq:pagarank-system}, some other approaches to compute $p$ have been investigated and compared to the power method, as in \cite{DelCorso,Gleich,GG06,Tudisco2011}.   Therefore Theorem \ref{lin-sis} provides a useful tool as it  allows to approach in a similar way many other  (large scale) problems of the type \eqref{eq:prk}, and reveals a further relation between ergodicity coefficients and the Pagerank centrality. 

\bibliographystyle{plain} 
\bibliography{../../../../library}
\end{document}